\documentclass[10pt,twoside,english,reqno,a4paper]{amsart}
\usepackage{listings,graphicx,amsmath,varioref,amscd,amssymb,color,bm,stmaryrd}
\usepackage[colorlinks=false]{hyperref}
\usepackage{amsthm}
\newlength{\defbaselineskip}
\usepackage{amssymb}
\usepackage{amsmath}
\usepackage{amsthm}
\usepackage{amscd}
\usepackage{amssymb}
\usepackage{amsfonts}
\usepackage{graphics}

\usepackage{hyperref}
\usepackage{graphics}

\usepackage[english]{babel}

\theoremstyle{plain}
\newtheorem{theorem}{Theorem}[section]

\newtheorem{lemma}[theorem]{Lemma}

\theoremstyle{definition}
\newtheorem{defin}[theorem]{Definition}

\newtheorem{remark}[theorem]{Remark}

\theoremstyle{remark}

\long\def\salta#1{\relax}

\def\rn{\mathbb{R}^{N}}

\def\be{\begin{equation}}
\def\ee{\end{equation}}
\def\rife#1{(\ref{#1})}

\def\vare{\varepsilon}

\def\into{\int_{\Omega}}

\def\w-1p'{W^{-1,p'}(\Omega)}

\def\w-1pd{W^{-1,p'}(D)}
\def\pw-1p'{L^{p'}(0,T;W^{-1,p'}(\Omega))}
\def\l{\textsl{L}}

\def\dys{\displaystyle}

\def\lp'n{(L^{p'}(\Omega))^{N}}

\def\huz{H^1_0 (\Omega)}

\def\luo{L^{1}(\Omega)}

\def\car#1{\raise2pt\hbox{$\chi$}_{#1}}

\definecolor{seagreen}{rgb}{0.18, 0.55, 0.34}

\def\lp'n{(L^{p'}(\Omega))^{N}}

\def\huz{H^1_0 (\Omega)}

\def\t1p0{T^{1,p}_{0}(\Omega)}
\def\w-1p'{W^{-1,p'}(\Omega)}
\def\pw-1p'{L^{p'}(0,T;W^{-1,p'}(\Omega))}

\def\lil2{L^{\infty}(0,T;L^2 (\Omega))}

\def\l2h10{L^2 (0,T ; H^1_0 ( \Omega ))}

\def\m2{M^{\frac{N(p-1)}{N-1}}(\Omega)}

\def\l{\textsl{L}}

\def\vare{\varepsilon}

\def\into{\int_{\Omega}}

\def\intq{\int_{Q}}

\def\intq1{\displaystyle \int_{\Omega \times (0, 1)}}

\def\bk{\color{black}}

\def\rife#1{(\ref{#1})}

\def\dys{\displaystyle}

\def\og{\leavevmode\raise.3ex\hbox{$\scriptscriptstyle\langle\!\langle$~}}
\def\fg{\leavevmode\raise.3ex\hbox{~$\!\scriptscriptstyle\,\rangle\!\rangle$}}

\author[F. Oliva]{Francescantonio Oliva}
\author[F. Petitta]{Francesco Petitta}

\address[F. Oliva]{Dipartimento di Scienze di Base e Applicate
per l' Ingegneria, ``Sapienza", Universit\`a di Roma, Via Scarpa 16, 00161 Roma, Italy. 
\\ \url{francesco.oliva@sbai.uniroma1.it}}
\address[F. Petitta]{Dipartimento di Scienze di Base e Applicate
per l' Ingegneria, ``Sapienza", Universit\`a di Roma, Via Scarpa 16, 00161 Roma, Italy. 
\\ \url{francesco.petitta@sbai.uniroma1.it}}


\begin{document}

\title{On singular elliptic equations with measure sources}





\begin{abstract}

We prove existence of solutions for a class of singular elliptic problems with a general measure as source term whose model is
$$\begin{cases}
 \dys   -\Delta u  = \frac{f(x)}{u^{\gamma}} +\mu & \text{in}\ \Omega,\\
 u=0 &\text{on}\ \partial\Omega,\\
  u>0 &\text{on}\ \Omega,
  \end{cases}
$$
where $\Omega$ is an open bounded subset of $\rn$. Here $\gamma > 0$, $f$ is a nonnegative function on $\Omega$, and $\mu$ is a nonnegative bounded Radon measure on $\Omega$. 

\vskip 0.5\baselineskip

%


\end{abstract}

\maketitle




\section{Introduction}

In this paper we prove existence of a nonnegative weak solution to the following elliptic problem 
\begin{equation}
\begin{cases}
 \displaystyle -\operatorname{div}(A(x)\nabla u) = \frac{f(x)}{u^\gamma}+\mu &  \text{in}\, \Omega, \\
 u=0 & \text{on}\ \partial \Omega,
\label{pb}
\end{cases}
\end{equation}
where $\Omega$ is an open bounded subset of $\rn$, $N\ge 2$, $\gamma > 0$, $f$ is a nonnegative function which has some Lebesgue regularity and $\mu$ is a nonnegative bounded Radon measure.  Here $A:\Omega\to \mathbb{R}^{N^2}$ is a bounded elliptic matrix, and we ask both $f$ and $\mu$ to be not identically zero. We stress that the problem is singular as one ask to the solution to be zero on the boundary. 
\\ If $\mu\equiv 0 $, $A(x)\equiv I$, and $f$ is  smooth, then problem \rife{pb} has  been extensively studied in the past (see for instance \cite{crt,lm} and references therein).  Recently, the existence of a distributional solution for problem \rife{pb} has been considered in \cite{bo}  again in the homogeneous  case $\mu\equiv 0$ (see also \cite{boca,am} for further improvements). 

The nonhomogeneous case (i.e. $\mu\neq 0$) has also been considered: motivated by the study of $G$-convergence in periodic composite media, in \cite{gmm,gmm2} the authors prove existence of bounded  solutions  to problems \rife{pb} if  both $f$ and  $\mu$ belong to $ L^{q}(\Omega)$, with $q>\frac{N}{2}$,  while in \cite{cgs} symmetry properties for solutions  of a related semilinear problem are considered (see also \cite{cd,c,cc} for further related results).\bk

Let us also observe that problem \rife{pb} is also related, through the  standard Hopf-Lax type transformation $u=v^{1- \eta}$ (with $\eta=\frac{\gamma}{\gamma+1}$) to the singular quadratic problem (for simplicity, consider $A(x)=I$)  
\begin{equation}
\begin{cases}
 \displaystyle -\Delta v +\eta\frac{|\nabla v|^{2}}{v} = \tilde{g}(x)v^{\eta}+ {\tilde{f}(x)}\,&  \text{in}\ \Omega, \\
 v=0 & \text{on}\ \partial \Omega\,,
\label{pbq}
\end{cases}
\end{equation}
with $\tilde{g}=\frac{\mu}{1-\eta}$ and $\tilde{f}=\frac{f}{1-\eta}$. If $\tilde{g}\equiv 0 $ and $f$ is an $L^{1}$ function then problem \rife{pbq} has been recently considered in the literature (see for instance \cite{b, a6,  gps} and references therein).  
Finally, the case of $\tilde{g}$  being a constant and $\tilde{f} \in L^{\frac{2N}{N+2}}(\Omega)$ has been studied in \cite{acm} through variational methods.   In particular, Theorem \ref{main} below  can be viewed  as a generalization of  the existence result in \cite[Theorem 1.2]{acm}. 

\medskip

Our aim is to give a complete account on the solvability of boundary value problems as  \rife{pb} under minimal assumptions on the data.  
The plan of the paper is the following:    Section \ref{secgen} is devoted to  the proof of existence of weak solutions for problem \rife{pb}  in the general case through an approximation argument. In Section \ref{secregular} we  further investigate a special case of \rife{pb}  (i.e. $A(x)=I$ and $\gamma<1$) on smooth domains; in this case we prove existence and uniqueness of suitable solutions to \rife{pb}.  Finally, in Section \ref{secintrop} we show how to generalize some of the results we obtained to the case of a nonlinear principal part of Leray-Lions type with growth $p-1$ as the $p$-laplacian. \bk 

\medskip

\subsection*{Notations.}  If no otherwise specified, we will denote by $C$ several constants whose value may change from line to line and, sometimes, on the same line. These values will only depend on the data (for instance $C$ can depend on $\Omega$, $\gamma$, $N$, $\alpha$ and $\beta$) but they will never depend on the indexes of the sequences we will often introduce. For the sake of simplicity we will often use the simplified notation $$\into f:=\into f(x)\ dx\,,$$ when referring to integrals when no ambiguity on the  variable of integration   is possible.

For fixed $k>0$ we will made use of the truncation functions $T_{k}$ and $G_{k}$ defined as
$$
T_k(s)=\max (-k,\min (s,k)), 
$$
and
$$
G_k(s)=(|s|-k)^+ \operatorname{sign}(s). 
$$

We will also make use of the Marcinkievicz space $M^q(\Omega)$ (or weak $L^{q}(\Omega)$) which is defined, for every $0<q<\infty$, as the space of all measurable functions $f$ such that there exists $c>0$ with 
$$m(\{x:|f(x)|>t\}) \le \frac{c}{t^q}.$$
Of course, as $\Omega$ is bounded, then it suffices for the previous inequality to hold for $t\geq t_{0}$, for some fixed positive $t_{0}$. 
We only recall that the following continuous embeddings hold
\begin{equation} L^{q}(\Omega) \hookrightarrow M^q(\Omega) \hookrightarrow L^{q-\epsilon}(\Omega),\label{marcin}\end{equation}
for every $1<q<\infty$ and $0<\epsilon\leq q-1$.

\section{Main assumptions and existence of a solution}
\label{secgen}

Let us  consider the following boundary value problem
\begin{equation}
\begin{cases}
 \displaystyle -\operatorname{div}(A(x)\nabla u) = \frac{f(x)}{u^\gamma}+\mu &  \text{in}\, \Omega, \\
 u=0 & \text{on}\ \partial \Omega,
\label{pba}
\end{cases}
\end{equation}
where $\Omega$ is any open bounded subset of $\rn$, $N\ge 2$, $\gamma > 0$, $f$ is a nonnegative function in $L^{1}(\Omega)$, $\mu$ is a nonnegative bounded Radon measure, and   $A:\Omega\to \mathbb{R}^{N^2}$ is a bounded elliptic matrix, which satisfies  
$$
\alpha |\xi|^2 \le A(x)\xi \cdot \xi,\ \ \ |A(x)|\le \beta,
$$
for every $\xi$ in $\mathbb{R}^N$, for almost every $x$ in $\Omega$, for some  $0<\alpha\le\beta$. As we said we ask to both $f$ and $\mu$ to be not identically zero otherwise we fall back in the previous considered cases of, respectively, \cite{s} and \cite{bo}. 

In this kind of generality, uniqueness of weak solutions  is not expected in general even in smoother cases (we refer to \cite{bc} \bk for further comments on this fact).  In Section \ref{secregular} below we will prove a uniqueness result in a very special model case (i.e. $A(x)=I$ and $\gamma<1$).
 
Our aim is to prove the  existence of  suitable weak solutions to  \rife{pba}.  
Here is the notion of solution we will consider. 
\begin{defin}
If $\gamma\leq 1$, then a weak solution to problem \rife{pba} is a function $u\in W^{1,1}_0(\Omega)$  such that 
 \begin{equation} \displaystyle \forall \omega \subset\subset \Omega \ \exists c_{\omega} : u \ge c_{\omega}>0,\label{priweakdef<=1}\end{equation}
and such that
 \begin{equation} \displaystyle \int_{\Omega}A(x)\nabla u \cdot \nabla \varphi =\int_{\Omega} \frac{f\varphi}{u^\gamma} + \int_{\Omega} \varphi d\mu,\ \ \ \forall \varphi \in C^1_c(\Omega).\label{secweakdef<=1}\end{equation}
 If $\gamma>1$, then a weak solution to problem \rife{pba} is a function $u\in W^{1,1}_{loc}(\Omega)$  satisfying \rife{priweakdef<=1}, \rife{secweakdef<=1},  and such that $T_k^{\frac{\gamma+1}{2}}(u)\in H^1_0(\Omega)$ for every fixed $k>0$.
\label{weakdef<=1}
\end{defin}
\begin{remark} First of all observe that, due to \rife{priweakdef<=1}, all terms in \rife{secweakdef<=1} are well defined. 
Let us spend some words on how the boundary data is intended in the strongly singular case (i.e. $\gamma>1$). In fact, in this case  the Dirichlet datum is not expected to be achieved in the classical sense of traces. Even in the case $\mu=0$, this was already noticed  in \cite{lm}; in this paper the authors proved that, if $f$ is a bounded smooth function then no solutions in $H^{1}_{0}(\Omega)$ can be found if $\gamma\geq 3$. The threshold $3$ is essentially  due to the fact that the datum $f$ is assumed to be bounded. A sharper result was recently proven, through variational methods,  in \cite[Theorem 1]{sz} for a nonnegative $f\in\luo$ and $A(x)=I$. This result reads as: an $H^{1}_{0}(\Omega)$ solution does exist for problem
$$
\begin{cases}
 \displaystyle -\Delta u= \frac{f(x)}{u^\gamma}&  \text{in}\, \Omega, \\
 u=0 & \text{on}\ \partial \Omega,
\end{cases}
$$
 if and only if there exist a function  $u_{0}\in \huz$ such that 
 \begin{equation}\label{cin}
 \int_{\Omega}fu_{0}^{1-\gamma}\ dx<\infty\,.
 \end{equation}
 
 A straightforward computation shows that, at least in the radial case, for any $\gamma>1$ condition \rife{cin} is not satisfied for suitable $f\in\luo$.   
For that reason we need to  impose the weaker condition on the truncations of $u$ (i.e. $T(u)^{\frac{\gamma+1}{2}}\in\huz$) in order to give a (weak) sense to the boundary datum. This kind of weak boundary condition was already used in the literature (see for instance \cite{bo}). Here we need to use truncations of $u$ as the presence of  the measure data $\mu$ does not allow to conclude that $u^{\frac{\gamma+1}{2}}$ itself belongs to $\huz$.  Anyway let us observe that, if $\Omega$ is smooth enough and  $v^{\frac{\gamma+1}{2}}\in \huz$  ($v$ being  a nonnegative function), then
$$
\lim_{\vare\to 0^{+}}\frac{1}{\vare}\int_{\{x:{\rm dist}(x,\partial\Omega)<\vare\}} v^{\frac{\gamma+1}{2}} (x) \, dx=0
$$   
(see \cite{po}). As $\gamma>1$, using H\"older's inequality one can easily get that 
$$
\lim_{\vare\to 0^{+}}\frac{1}{\vare}\int_{\{x:{\rm dist}(x,\partial\Omega)<\vare\}} v (x) \, dx=0,
$$
which is a clearer way to understand the boundary condition we use here.

\bk
\end{remark}

We will prove existence of solutions for problem \rife{pba} by approximation. In order to do that we need some preliminary results on the approximating sequences of solutions that can be proven no matter of the value of $\gamma$.

Let us consider the following problem
\begin{equation}
\begin{cases}
 \displaystyle -\operatorname{div}(A(x)\nabla u_n) =\frac{f_n}{(u_n+ \frac{1}{n})^\gamma}+\mu_n &  \text{in}\, \Omega, \\
 u_n=0 & \text{on}\ \partial \Omega,
\label{eqapprox}
\end{cases}
\end{equation}
where $f_n$ is the truncation at level $n$ of $f$ and $\mu_n$ is a sequence of smooth nonnegative functions, bounded in $L^1(\Omega)$, converging weakly to $\mu$ in the sense of the measures. 
\\We want to prove the existence of a weak solution of problem \rife{eqapprox} for every fixed $n \in \mathbb{N}$.
\begin{lemma}
Problem \rife{eqapprox} admits a nonnegative weak solution $u_n \in H^1_0(\Omega)\cap L^{\infty}(\Omega)$. 
\label{soleqapprox}
\end{lemma}
\begin{proof}
The proof is based on standard  Schauder's fixed point argument. 
Let $n\in\mathbb{N}$ be fixed, let us  define
$$G:L^2(\Omega) \to L^2(\Omega)\,,$$
as the map that, for any $v \in L^2(\Omega)$ gives  the weak solution  $w$ to the following problem
$$
\begin{cases}
 \displaystyle -\operatorname{div}(A(x)\nabla w) =\frac{f_n}{(|v|+ \frac{1}{n})^\gamma}+\mu_n &  \text{in}\, \Omega, \\
 w=0 & \text{on}\ \partial \Omega.
\end{cases}
$$
It follows from classical theory (e.g. by Lax-Milgram lemma) that $w \in H^1_0(\Omega)$ for every fixed $v \in L^2(\Omega)$. This implies that we can choose $w$ as test function in the weak formulation.
\\Thus
 $$ \displaystyle \alpha\int_{\Omega}|\nabla w|^2 \le\int_{\Omega}A(x)\nabla w \cdot \nabla w =\int_{\Omega} \frac{f_nw}{(|v|+\frac{1}{n})^\gamma} + \int_{\Omega} w \mu_n \le \left (n^{\gamma+1}+C(n)\right ) \int_{\Omega}|w|.$$
Therefore, using the Poincar\'e inequality on the left hand side and the H\"older inequality on the right side
 $$ \displaystyle\int_{\Omega}|w|^2  \le C'\left (n^{\gamma+1}+C(n)\right ) \left ( \int_{\Omega}|w|^2 \right )^{\frac{1}{2}}.$$
This means that
$$||w||_{L^2(\Omega)}\le C'\left (n^{\gamma+1}+C(n)\right ),$$
where $C',C(n)$ are independent from $v$. Thus, we have that the ball of radius $C'\left (n^{\gamma+1}+C(n)\right )$ is invariant for $G$. 
Now we prove that the map $G$ is continuous in $L^2(\Omega)$. Let us choose a sequence $v_k$ that converges to $v$ in $L^2(\Omega)$; then, by the dominated convergence theorem
$$\displaystyle \left (\frac{f_n}{(|v_k|+ \frac{1}{n})^\gamma}+\mu_n\right)_{k \in \mathbb{N}} \text{ converges to  } \displaystyle \left (\frac{f_n}{(|v|+ \frac{1}{n})^\gamma}+\mu_n\right) \text{  in }L^2(\Omega).$$ 
Thus, by the uniqueness of the weak solution for the linear problem, we can say that $w_k:=G(v_k)$ converges to $w:=G(v)$ in $L^2(\Omega)$. Therefore, we proved that $G$ is continuous. 

What finally need to check that the set $G(L^2(\Omega))$ is relatively compact in $L^2(\Omega)$. We proved before that
$$\into |\nabla G(v)|^{2}\le C(n,\gamma),$$
for any $v\in L^{2}(\Omega)$, so that, $G(v)$ is relatively compact in $L^{2}(\Omega)$ by Rellich-Kondrachov theorem. 

Now we can finally apply Schauder fixed point theorem to obtain  that $G$ has a fixed point $u_n \in L^2(\Omega)$ that is a  solution to \rife{eqapprox} in $H^1_0(\Omega)$. Moreover,  $u_n$ belongs to $L^\infty(\Omega)$ (see \cite{s}). Here $\left(\frac{f_n}{(|u_n|+ \frac{1}{n})^\gamma}+\mu_n\right)\ge 0$ then the maximum principle implies that $u_n\ge0$ and this concludes the proof.
\end{proof}

The next step consists in the proof that $u_n$ is uniformly bounded from below on the compact subsets of $\Omega$.
\begin{lemma}
The sequence $u_n$ is such that for every $\omega \subset\subset \Omega$ there exists $c_{\omega}$ (not depending on $n$) such that 
$$ \displaystyle u_n(x) \ge c_{\omega}>0,\text{   for a.e. x in $\omega$, and for every n in $\mathbb{N}$.}$$ 
\end{lemma}
\begin{proof}
Let consider the following problem
\begin{equation}
\begin{cases}
 \displaystyle -\operatorname{div}(A(x)\nabla v_n) =\frac{f_n}{(v_n+ \frac{1}{n})^\gamma} &  \text{in}\, \Omega, \\
 v_n=0 & \text{on}\ \partial \Omega.
\label{eqboapprox}
\end{cases}
\end{equation}
It was proved in \cite{bo} the existence of a weak solution $v_{n}$ of \rife{eqboapprox} such that
$$ \displaystyle \forall \omega \subset\subset \Omega \ \exists c_{\omega} : v_n \ge c_{\omega}>0,$$
for almost every $x$ in $\omega$ and where $c_{\omega}$ is independent of $n$. Thus, taking $(u_n-v_n)^-$ as test function in the difference between the formulations of, respectively,  \rife{eqapprox} and \rife{eqboapprox}, we obtain, using ellipticity
$$ 
\begin{array}{l}\displaystyle -\alpha\into|\nabla (u_n-v_n)^-|^{2}\geq \int_{\Omega}A(x) \nabla (u_n-v_n) \cdot  \nabla (u_n-v_n)^- \\ \\ \dys = \int_{\Omega} \left(\frac{f_n(x)}{(u_n+\frac{1}{n})^\gamma} \dys -\frac{f_n(x)}{(v_n+\frac{1}{n})^\gamma}\right)(u_n-v_n)^- + \int_{\Omega} \mu_n(u_n-v_n)^- \ge 0,\end{array}$$
that implies $u_{n }\geq v_{n}$ for a.e.  $x$ in $\omega$, and so  
$$
 \displaystyle \forall \omega \subset\subset \Omega \ \exists c_{\omega} : u_n\ge c_{\omega}>0, \ \ \text{for a.e.  $x$ in $\omega$}.
$$
\end{proof}

We can now prove existence of solutions for problem \rife{pba}. In order to do that we distinguish between two cases. 

\subsection{The case $\gamma\leq 1$}

As one could expect from classical theory of measure data problem we will have that $u_n$ is bounded in $W^{1,q}_0(\Omega)$ for every $q<\frac{N}{N-1}$. 

\begin{lemma}
 Let $u_n$ be the solution of \rife{eqapprox} with $\gamma\leq 1$. Then $u_n$ is bounded in $W^{1,q}_0(\Omega)$ for every $q<\frac{N}{N-1}$.
\label{bounded<1}
\end{lemma}
\begin{proof}
We follow the classical approach of \cite{b6}. We will prove that  $\nabla u_n$ is bounded in $M^{\frac{N}{N-1}}(\Omega)$. Without loss of generality we can suppose $N>2$. The case $N=2$ is easier and can be treated in a standard way with many simplifications. 

We have 
$$
\begin{array}{l}\displaystyle \{|\nabla u_n| \ge t\} = \{|\nabla u_n| \ge t, u_n < k\} \cup \{|\nabla u_n| \ge t, u_n \ge k\}
\\\\ \subset \{|\nabla u_n| \ge t, u_n < k\} \cup \{u_n \ge k\}\,,\end{array}$$
thus,
$$\displaystyle m(\{|\nabla u_n| \ge t\}) \le m(\{|\nabla u_n| \ge t, u_n < k\})+ m(\{u_n \ge k\}).$$
\\ \\We take $\varphi = (T_k(u_n) +\epsilon)^{\gamma}-\epsilon^{\gamma}$ as test function in the weak formulation of \rife{eqapprox} where $\epsilon>0$. We stress that the $\epsilon$ perturbation of $T_k(u_n)^{\gamma}$ is only needed in the case $\gamma<1$ in order to be allowed to use it as test function.  
We obtain,
\begin{equation}
 \displaystyle \alpha \int_{\Omega}|\nabla T_k(u_n)|^2 (T_k(u_n) +\epsilon)^{\gamma-1} \le \int_{\Omega} \frac{f_n[(T_k(u_n) +\epsilon)^{\gamma}-\epsilon^{\gamma}]}{(u_n+\frac{1}{n})^{\gamma}} + \int_{\Omega} [ (T_k(u_n) +\epsilon)^{\gamma}-\epsilon^{\gamma}] \mu_n.
\label{stima1<1}
\end{equation}
For fixed $n$, if $\epsilon$ is taken smaller than $\displaystyle \frac{1}{n}$, then we can say that $\displaystyle \frac{[(T_k(u_n) +\epsilon)^{\gamma}-\epsilon^{\gamma}]}{(u_n+\frac{1}{n})^\gamma}\le 1$. Thus,  for the right hand side of \rife{stima1<1} we have both
 $$ \displaystyle \int_{\Omega} \frac{f_n[(T_k(u_n) +\epsilon)^{\gamma}-\epsilon^{\gamma}]}{(u_n+\frac{1}{n})^{\gamma}} \le  \int_{\Omega} f_n \le ||f||_{L^1(\Omega)}\,,$$
and
 $$ \displaystyle \int_{\Omega} [(T_k(u_n) +\epsilon)^{\gamma}-\epsilon^{\gamma}] \mu_n \le  (k+\epsilon)^\gamma||\mu_n||_{L^1(\Omega)}.$$
Combining the previous results we obtain
 $$ \displaystyle \int_{\Omega}|\nabla T_k(u_n)|^2 (T_k(u_n) +\epsilon)^{\gamma-1} \le C[1+(k+\epsilon)^\gamma].$$
Therefore, 
 $$\begin{array}{l} \displaystyle \int_{\Omega}|\nabla T_k(u_n)|^2 = \int_{\Omega}\frac{|\nabla T_k(u_n)|^2}{(T_k(u_n) +\epsilon)^{\gamma-1}} (T_k(u_n) +\epsilon)^{\gamma-1}\\\\  \displaystyle \le (k+\epsilon)^{1-\gamma}||\nabla T_k(u_n)|^2 (T_k(u_n) +\epsilon)^{\gamma-1}||_{L^1(\Omega)}
\leq C(k+\epsilon)^{1-\gamma}[1+(k+\epsilon)^{\gamma}],\end{array}$$
and passing to the limit in $\epsilon$
 \begin{equation}\label{stimale} \displaystyle \int_{\Omega}|\nabla T_k(u_n)|^2 \le C(k^{1-\gamma}+k).\end{equation}
Let us observe that the constant $C$ does not depend on the index $n$ of the sequence.
Thus, it follows from the Sobolev inequality that 
 $$ \displaystyle\frac{1}{\mathcal{S}^2} \left (\int_{\Omega}|T_k(u_n)|^{2^*}\right)^\frac{2}{2^*}\le \int_{\Omega}|\nabla T_k(u_n)|^2 \le C(k^{1-\gamma}+k).$$
If we restrict the integral on the left hand side on $\{u_n \ge k\}$ (on which $T_k(u_n)=k$) 
we then obtain
$$\displaystyle k^2 m(\{u_n \ge k\})^\frac{2}{2^*} \le C(k^{1-\gamma}+k),$$
so that
$$\displaystyle m(\{u_n \ge k\}) \le C\left( \frac{ (k^{1-\gamma}+k)}{k^2}\right ) ^{ \frac{N}{N-2}} \le\frac{C}{k^{\frac{N}{N-2}}}\ \ \ \ \forall k\ge 1,$$ 
that is $u_n$ is bounded in $M^{\frac{N}{N-2}}(\Omega).$
\\ \\
Now we estimate $m(\{|\nabla u_n| \ge t, u_n < k\})$. From \rife{stimale} we have
$$\displaystyle m(\{|\nabla u_n| \ge t, u_n < k\}) \le \frac{1}{t^2}\int_{\Omega}|\nabla T_k(u_n)|^2 \le \frac{C(k^{1-\gamma}+k)}{t^2}\le\frac{Ck}{t^2}\ \ \ \ \forall k\ge 1.$$
Combining previous inequalities we obtain
$$\displaystyle m\left(\{|\nabla u_n| \ge t\}\right) \le m\left(\{|\nabla u_n| \ge t, u_n < k\}\right) + m\left(\{u_n \ge k\}\right) \le  \frac{Ck}{t^2}+ \frac{C}{k^{\frac{N}{N-2}}}\ \ \ \ \forall k\ge 1.$$
We then choose $\displaystyle k=t^{\frac{N-2}{N-1}}$
$$\displaystyle m\left(\{|\nabla u_n| \ge t\}\right) \le  \frac{Ct^{\frac{N-2}{N-1}}}{t^2}+ \frac{C}{t^{\frac{N}{N-1}}}\ \ \ \ \forall t\ge 1.$$
Thus we get for a constant $C$ not depending on $n$
$$\displaystyle m\left(\{|\nabla u_n| \ge t\}\right) \le \frac{C}{t^{\frac{N}{N-1}}}\ \ \ \ \forall t\ge 1.$$
We have just proved that $\nabla u_n$ is bounded in $M^{\frac{N}{N-1}}(\Omega)$. This implies by property \rife{marcin} that $u_n$ is bounded in $W^{1,q}_0(\Omega)$ for every $q<\frac{N}{N-1}$.
\end{proof}

We now can prove our existence result for $\gamma\leq 1$. 

\begin{theorem}
Let $\gamma\leq 1$. Then there exists a weak solution $u$ of \rife{pba} in $W^{1,q}_0(\Omega)$ for every $q<\frac{N}{N-1}.$
\label{esistenzagamma<1}
\end{theorem}
\begin{proof}
It follows from Lemma \ref{bounded<1} that there exists $u$ such that (up to not relabeled subsequences) the sequence $u_n$ converges  weakly to $u$ in $W^{1,q}_0(\Omega)$ for every in $q<\frac{N}{N-1}$. 
This implies that for $\varphi$ in $C^1_c(\Omega)$
$$\displaystyle \lim_{n\to +\infty} \int_{\Omega}A(x)\nabla u_n \cdot \nabla \varphi =\int_{\Omega}A(x)\nabla u \cdot \nabla \varphi.$$
Moreover, by compact embeddings, we can also assume that $u_n$ converges to $u$ both strongly  in $L^{1}(\Omega)$ and a.e.
Thus, taking $\varphi$ in $C^1_c(\Omega)$, we have that 
$$\displaystyle 0\le \left|\frac{f_n\varphi}{(u_n+\frac{1}{n})^\gamma} \right | \le  \frac{||\varphi||_{L^{\infty}(\Omega)}}{c_{\omega}}f,$$
where $\omega$ is the set $\{\varphi\not= 0\}.$ This is enough to apply the dominated convergence theorem so that
$$\displaystyle \lim_{n\to +\infty} \int_{\Omega}\frac{f_n\varphi}{(u_n+\frac{1}{n})^\gamma} = \int_{\Omega}\frac{f\varphi}{u^\gamma}.$$
This concludes the  proof of the result as it is straightforward to pass to the limit in the last term involving $\mu_{n}$.
\end{proof}

\subsection{The strongly singular case: $\gamma > 1$}

In this case, as we already mentioned, we can only hold some {\it local} estimates on $u_n$ in the Sobolev space. In order to give sense to the function $u$ on the boundary of $\Omega$,  at least a weak sense, we shall provide {\it global} estimates on $T_k^{\frac{\gamma+1}{2}}(u_n)$ in $\huz$. 

As in the previous case we consider the solutions of the approximating solutions of  \rife{eqapprox}. Here is our global estimate on the power of the truncation of $u_{n}$. 
\begin{lemma}
Let $u_n$ be the solution of \rife{eqapprox} with $\gamma>1$. Then $T_k^{\frac{\gamma+1}{2}}(u_n)$ is bounded in $H^1_0(\Omega)$ for every fixed $k>0$.
\label{T_kbounded}
\end{lemma}
\begin{proof} 
We take $\varphi = T_k^\gamma(u_n)$ as a test function in \rife{eqapprox} and we have
\begin{equation}
\displaystyle \gamma\int_{\Omega}A(x)\nabla u_n \cdot \nabla T_k(u_n) T_k^{\gamma-1}(u_n) =   \int_{\Omega} \frac{f_nT_k^\gamma(u_n)}{(u_n+\frac{1}{n})^{\gamma}} + \int_{\Omega} T_k^\gamma(u_n)\mu_n.
\label{stimalemma1>1}
\end{equation}
Using ellipticity we can estimate  the term on the left hand side of \rife{stimalemma1>1} as 
$$ \displaystyle \gamma\int_{\Omega}A(x)\nabla u_n \cdot \nabla T_k(u_n) T_k^{\gamma-1}(u_n) \ge \displaystyle \alpha\gamma\int_{\Omega}|\nabla T_k^{\frac{\gamma+1}{2}}(u_n)|^2.$$
Recalling that $\frac{T_k^\gamma(u_n)}{(u_n+\frac{1}{n})^{\gamma}} \le  \frac{u_n^\gamma}{(u_n+\frac{1}{n})^{\gamma}}\le 1,$ then for the term on the right hand side of \rife{stimalemma1>1} we can write
 $$ \displaystyle \int_{\Omega} \frac{f_nT_k^\gamma(u_n)}{(u_n+\frac{1}{n})^{\gamma}} + \int_{\Omega} T_k^\gamma(u_n)\mu_n \le 1+ k^\gamma \le C,$$
so that combining the previous inequalities we get
$$ \displaystyle \int_{\Omega}|\nabla T_k^{\frac{\gamma+1}{2}}(u_n)|^2 \le C,$$
that is what we had to prove. 
\end{proof}
Now, in order to pass to the limit in the weak formulation, we only need to prove some local estimates on $u_{n}$.  We prove the following
\begin{lemma}
Let $u_n$ be the solution of \rife{eqapprox} with $\gamma>1$. Then $u_n$ is bounded in $W^{1,q}_{loc}(\Omega)$ for every $q<\frac{N}{N-1}$.
\label{stima1>1}
\end{lemma}
\begin{proof}

We divide the proof in two steps.

{\bf Step $1$.} $G_1(u_n)$ is bounded in $W^{1,q}_{0}(\Omega)$ for every $q<\frac{N}{N-1}$.

We need to prove that
\begin{equation}\label{claim1} \displaystyle \int_{\{u_n>1\}}|\nabla u_n|^q \le C,\end{equation}
where $q<\frac{N}{N-1}$.
Analogously to the case $\gamma <1$, we have to prove that $\nabla G_1(u_n)$ is bounded in $M^{\frac{N}{N-1}}(\Omega)$. We have
\begin{equation}
m(\{|\nabla u_n|>t, u_n>1\}) \le m(\{|\nabla u_n|>t, 1<u_n\le k\}) +  m(\{u_n > k\}).
\label{cavalieri>1}
\end{equation}
In order to estimate \rife{cavalieri>1} we take $\varphi = T_k(G_1(u_n)) \ (k>1)$ as a test function in \rife{eqapprox}.

Recalling that $\nabla T_k(G_1(u_n))=\nabla u_n$ only when  $1<u_{n}\leq k$ (otherwise is zero), and that  $T_k(G_1(u_n))=0$ on $\{u_{n}\leq 1\}$, we have 
$$ \displaystyle \alpha \int_{\Omega}|\nabla T_k(G_1(u_n))|^2 \le   \int_{\Omega} \frac{f_nT_k(G_1(u_n))}{(1+\frac{1}{n})^{\gamma}} + \int_{\Omega} T_k(G_1(u_n)) \mu_n \le Ck\,,$$
and 
$$\begin{array}{l} \displaystyle \int_{\Omega}|\nabla T_k(G_1(u_n))|^2 = \int_{\{1<u_n\le k\}}|\nabla u_n|^2 \\\\ \dys \ge \int_{\{|\nabla u_n|>t, 1<u_n\le k\}}|\nabla u_n|^2\ge t^{2} m(\{|\nabla u_n|>t, 1<u_n\le k\}),\end{array}$$
so that, 
$$m(\{|\nabla u_n|>t, 1<u_n\le k\})\le \frac{Ck}{t^2} \ \ \ \forall k \ge 1\,.$$
By keeping track of the dependence on $k$ in the proof of Lemma \ref{T_kbounded} one readily realize that
$$ \displaystyle \int_{\Omega}|\nabla T_k^{\frac{\gamma+1}{2}}(u_n)|^2 \le Ck^{\gamma},\ \text{for any}\ k>1,$$
which, reasoning as in the proof of Lemma \ref{bounded<1}, gives 
$$m(\{u_n > k\})\le \frac{C}{k^\frac{N}{N-2}} \ \ \ \forall k \ge 1.$$
In order to conclude we take again  $k=t^{\frac{N-2}{N-1}}$ and we get 
$$m(\{|\nabla u_n|>t, u_n>1\})\le \frac{C}{t^\frac{N}{N-1}} \ \ \ \forall t \ge 1,$$
and this proves \rife{claim1}.

{\bf Step $2$.} $T_1(u_n)$ is bounded in $H^1_{loc}(\Omega)$.

We need  to investigate the behavior of  $u_n$ for small values (namely $u_{n}\leq 1$).  We want to prove that for every $\omega \subset\subset \Omega$ 
\begin{equation}\label{claim2} \displaystyle \int_{\omega}|\nabla T_1(u_n)|^2 \le C.\end{equation}
We have already proved that $u_n \ge c_{\omega}$ on $\omega$.
We take $T_1^{\gamma}(u_n)$  as a test function in \rife{eqapprox} in order to obtain 
\begin{equation}
\displaystyle \int_{\Omega}A(x) \nabla u_n \cdot  \nabla T_1(u_n)T_1^{\gamma-1}(u_n) =\into \frac{f_n T_1^{\gamma}(u_n)}{(u_n+ \frac{1}{n})^{\gamma}} + \int_{\Omega} \mu_n T_1^{\gamma}(u_n)\le C\,.
\label{equazionelocale>1}
\end{equation}
Now observe that 
\begin{equation}
\displaystyle \int_{\Omega}A(x) \nabla u_n \cdot  \nabla T_1(u_n)T_1^{\gamma-1}(u_n) \ge \alpha \int_{\Omega}|\nabla T_1(u_n)|^2 T_1^{\gamma-1}(u_n) \ge \alpha c_{\omega}^{\gamma-1} \int_{\omega}|\nabla T_1(u_n)|^2.
\label{equazionelocale2>1}
\end{equation}
Combining \rife{equazionelocale>1} and \rife{equazionelocale2>1} we  get  \rife{claim2}.

The proof of Lemma \ref{stima1>1} is complete as $u_n = T_1(u_n) + G_1(u_n),$ then  $u_n$ is bounded in $W^{1,q}_{loc}(\Omega)$ for every $q<\frac{N}{N-1}$.  
\end{proof}
We can finally  state and prove the following existence result.
\begin{theorem}\label{lab}
Let $\gamma>1$. Then there exists a weak solution $u$ of $\rife{pba}$ in $W^{1,q}_{loc}(\Omega)$ for every $q<\frac{N}{N-1}$.
\end{theorem}
\begin{proof} Thanks to Lemma \ref{T_kbounded} and Lemma \ref{stima1>1} above the proof of Theorem \ref{lab} is just a straightforward readaptation of the one of  Theorem \ref{esistenzagamma<1}.
\end{proof}

\subsection{Regularity of the solution}
In this section we show how the solutions of problem \rife{pba} obtained in the previous sections  increase  their summability depending on the summability of the data.  

\medskip

We have the following
\begin{theorem}
Let $f\in L^m(\Omega)$ and let $\mu\in L^r(\Omega)$. Then there exists a solution $u$ to \rife{pba} such that: 
\begin{itemize}
\item [(i)]if $m,r>\frac{N}{2}$, then $u \in L^{\infty}(\Omega)$,
\item [(ii)]if $1\le m <\frac{N}{2}$, $r>\frac{N}{2}$, then $u \in L^{\frac{Nm(\gamma + 1)}{N-2m}}(\Omega)$,
\item [(iii)]if $m>\frac{N}{2}$, $1< r <\frac{N}{2}$,  then $u \in L^{r^{**}}(\Omega)$,
\item [(iv)]if $1\le m <\frac{N}{2}$, $1< r <\frac{N}{2}$, then $u \in L^{q}(\Omega)$, where $q=\min{(\frac{Nm(\gamma + 1)}{N-2m},r^{**})}$.\\
\end{itemize}
\label{teoreg}
\end{theorem}
\begin{proof} Let  $u_n$ be  the sequence of approximating solutions to  problem \rife{pba} introduced in \rife{eqapprox}.  Let $w_{n}$ be the sequence of  solutions of 
\begin{equation}
\begin{cases}
 \displaystyle -\operatorname{div}(A(x)\nabla w_n) = \mu_n &  \text{in}\, \Omega, \\
 w_n=0 & \text{on}\ \partial \Omega,
\label{s}
\end{cases}
\end{equation}
and let $v_{n}$ be the solutions of 
\begin{equation}
\begin{cases}
 \displaystyle -\operatorname{div}(A(x)\nabla v_n) = \frac{f_n(x)}{(v_n+\frac{1}{n})^\gamma} &  \text{in}\, \Omega, \\
 v_n=0 & \text{on}\ \partial \Omega,
\label{bo}
\end{cases}
\end{equation}
where $f_n$, $\mu_n$ are the truncation at level $n$ of both $f$ and $\mu$.

Let us  define $z_n=w_n+v_n$, thus we have
\begin{equation}
 \displaystyle -\operatorname{div}(A(x)\nabla z_n) = \frac{f_n(x)}{(v_n+\frac{1}{n})^\gamma} + \mu_n \ge \frac{f_n(x)}{(z_n+\frac{1}{n})^\gamma} + \mu_n .
 \label{zn}
\end{equation}
Taking $(z_n - u_n)^-$ as test function in the formulation both \rife{zn} and \rife{eqapprox}, and taking the difference between the two we obtain
$$
\displaystyle \int_{\Omega}A(x) \nabla (z_n-u_n) \cdot  \nabla (z_n-u_n)^- \ge \int_{\Omega} \left(\frac{f_n(x)}{(z_n+\frac{1}{n})^\gamma}-\frac{f_n(x)}{(u_n+\frac{1}{n})^\gamma}\right)(z_n-u_n)^-\ge 0\,,
$$
that implies 
$$
\displaystyle \int_{\Omega} |\nabla (z_n-u_n)^-|^2 \le 0\,.
$$
So we have 
$$
u_n\le z_n = w_n + v_n.
$$
Passing to the limit in $n$ we then have that the summability of the solution $u$ can not be worse than $q$,  where $q$ is the minimum between the summabilities of  $w$ and $v$ that are investigated, respectively, in  \cite{s} and  \cite{bo}. 
To show the optimality of this result what is left is the proof that 
$$
w_n\le u_n, \ \ \text{and}\ \ 		
v_n\le u_n.		
$$
The previous bounds can be found, reasoning as before, taking $(u_n - w_n)^-$ as test functions in  the difference between  \rife{eqapprox} and \rife{s}, and taking $(u_n - v_n)^-$, as test functions in  the difference between  \rife{eqapprox} and \rife{bo}.  
This completes the proof. 
\end{proof}
\begin{remark}
In  Theorem \ref{teoreg} we have explicitly omitted  the case where $r=1$. Let us observe that,  according with  the  Stampacchia regularity  result in \cite{s}, the statement holds true also in this case  provided we substitute   $1^{**}(=\frac{N}{N-2})$ with $1^{**}- \epsilon$, where $\epsilon$ is any strictly positive fixed number.
\end{remark}

\subsection{Further remarks in a model case}
\label{secregular}
 
In this section we consider a less general case, namely $A\equiv I$ and $\gamma<1$; we will show that something more can be said on suitable solutions to \rife{pba} in this Lazer-McKenna type model case.  

We consider  $\Omega$ to be  a bounded open subset of $\mathbb{R}^N (N \ge 2)$ with boundary $\partial \Omega$ of class $C^{2+\alpha}$ for some $0<\alpha<1$.   We consider the following semilinear elliptic problem 
\begin{equation}
	\begin{cases}
		\displaystyle -\Delta u = \frac{f(x)}{u^\gamma}+\mu &  \text{in}\, \Omega, \\
		u=0 & \text{on}\ \partial \Omega,
		\label{pblaplacian}
	\end{cases}
\end{equation}
where $0<\gamma<1$, and $\mu$ is a nonnegative bounded Radon measure on $\Omega$.

Concerning the data $f$, for a fixed $\delta>0$,  let 
$$
\Omega_{\delta}:=\{x\in\Omega: {\rm dist} (x,\partial\Omega)< \delta\}\,,
$$
and consider $f\in \luo\cap L^{\infty}(\Omega_{\delta})$ such that $f>0$ a.e. in ${\Omega}$. Here, with a little abuse of notation, we mean that $f$ is a purely $L^{1}$ function defined on $\Omega$ that is essentially bounded in a neighborhood of $\partial\Omega$. These assumptions can be regarded as a  suitable relaxation of the assumptions in \cite{lm}.

We shall  prove existence and uniqueness of a nonnegative weak solution to problem \rife{pblaplacian}. 
Observe that, to this aim,  the restriction to the case of the laplacian is not only technical. Also in the nonsingular case, uniqueness of distributional solutions for problems as \rife{pb} fails in general (see \cite{se} and \cite{boca} for further results in this direction).  
\\As we will see our argument is strongly based on the fact that, in this case,  a suitable solution to \rife{pblaplacian} can be found satisfying a further regularity property, namely the integrability of the lower order singular term of the equation. This fact is in general false if $\gamma\geq 1$ even in the case $\mu\equiv 0 $. This takes us to consider another notion of solution that is strictly related to the integrability of the singular term.
\medskip

Here is the notion of solution we will consider. 
\begin{defin}\label{defmain}
	A (weak) solution to problem \rife{pblaplacian} is a function $u\in L^1(\Omega)$  such that $u >0$ a.e.  in $\Omega$, $ {f}{u^{-\gamma}}  \in L^1(\Omega)$, and 
	\begin{equation} \displaystyle -\int_{\Omega}u \Delta \varphi =\int_{\Omega} \frac{f\varphi}{u^\gamma} + \int_{\Omega} \varphi d\mu,\ \ \ \forall \varphi \in C^2_0(\overline{\Omega}) .
		\label{weakdeflapl} 
	\end{equation} 
	\label{definition1}
\end{defin}
\begin{remark}
	We recall that by   $C^2_0(\overline{\Omega})$  we mean functions in $\varphi\in C^{2}(\overline{\Omega})$ with $\varphi=0$ on $\partial\Omega$. With this choice of test functions all  terms in \rife{weakdeflapl}  are well defined and  the boundary condition $u=0$ on $\partial\Omega$ is  contained in \rife{weakdeflapl}. Also observe that, due to the uniqueness result we will prove, a  solution $u$ in the above sense will coincide with  a solution in the sense of Definition \ref{priweakdef<=1} provided $ {f}{u^{-\gamma}}  \in L^1(\Omega)$. 
\end{remark}

We will prove the following

\begin{theorem}\label{main}
	Let $f$ be an a.e. positive function in $L^{1}(\Omega)\cap L^{\infty}(\Omega_{\delta})$, for some $\delta>0$. Then there exists a unique solution for problem \rife{pblaplacian} in the sense of Definition \ref{definition1}.  
\end{theorem}

In order to prove Theorem \ref{main} we need to prove first that a solution in the sense of Definition \ref{defmain} does exist. 
This can be easily checked by  sub and supersolutions method. 
We start by define the concept of  sub and supersolutions for problem \rife{pblaplacian}.
\begin{defin}  \label{subdef}
	A function $\underline{u}$ is a {\itshape subsolution} for \rife{pblaplacian} if $\underline{u} \in L^1(\Omega)$,  $\underline{u} >0$ in $\Omega$, $\displaystyle {f}{\underline{u}^{-\gamma}}  \in L^1(\Omega)$, and 
	\begin{equation}\label{subdefeq}
	\displaystyle -\int_{\Omega}\underline{u} \Delta \varphi \leq \int_{\Omega} \frac{f\varphi}{\underline{u}^\gamma} + \int_{\Omega} \varphi d\mu, \ \ \ \forall \varphi \in C^2_0(\overline{\Omega}), \ \varphi \ge 0.\end{equation}
	Analogously, we say that $\overline{u}$ is a {\it supersolution} for problem \rife{pblaplacian} if $\overline{u} \in L^1(\Omega)$,  $\overline{u} >0$ in $\Omega$, $\displaystyle {f}{\overline{u}^{-\gamma}}  \in L^1(\Omega)$,  and \rife{subdefeq} is satisfied with the opposite inequality sign (i.e. $\geq$).
\end{defin}

\medskip
We are ready to state the result that is the basis of the method of sub and supersolutions for problem \rife{pblaplacian}. The proof is suitable re-adaptation of an argument in \cite{mp} and we will only sketch it.

\begin{theorem}
	If problem \rife{pblaplacian} has a subsolution $\underline{u}$ and supersolution $\overline{u}$ with  $\underline{u} \le \overline{u}$ in $\Omega$, then there exists a solution $u$ to \rife{pblaplacian} (in the sense of Definition \ref{definition1}) such that $\underline{u} \le u \le \overline{u}$.
	\label{methodsubsup}
\end{theorem}
\begin{proof}[Sketch of the proof] 
We first considered a truncated problem with a modified nonlinearity $\tilde{g}(x,u)$. The goal consists in proving that the solution of this truncated problem turns out to solve problem \rife{pblaplacian}.

	We define $\tilde{g}:\Omega \times \mathbb{R} \to \mathbb{R}$ as
	$$\tilde{g}(x,t)=
	\begin{cases}
	\displaystyle {f(x)}{\underline{u}(x)^{-\gamma}} & \text{if} \ t < \underline{u}(x)  , \\
	\displaystyle {f(x)}{t^{-\gamma}} & \text{if} \   \underline{u}(x) \le t \le \overline{u}(x) , \\
	\displaystyle {f(x)}{\overline{u}(x)^{-\gamma}} & \text{if} \ t > \overline{u}(x) .
	\end{cases}$$ 
	Notice that by definition of sub and supersolution both $\displaystyle {f}{\underline{u}^{-\gamma}}$ and $\dys{f}{\overline{u}^{-\gamma}}$ belong to $L^1(\Omega)$. Moreover,  $\underline{u}>0$ so that, $\tilde{g}$ is well defined a.e. on $\Omega$ and for every fixed $v \in L^1(\Omega)$ we have that $\tilde{g}(x,v(x)) \in  L^1(\Omega)$.
	
First of all,  if $u$ satisfies	
	\begin{equation}
	\begin{cases}
	-\Delta u =  \tilde{g}(x,u)+\mu &  \text{in}\, \Omega, \\
	u=0 & \text{on}\ \partial \Omega,
	\label{regpb}
	\end{cases}
	\end {equation}
	
	\noindent then $\underline{u} \le u \le \overline{u}$. Thus $\tilde{g}(\cdot,u)=g(u)$, $\displaystyle {f}{u^{-\gamma}} \in L^1(\Omega)$, and $u$ is a solution to \rife{pblaplacian}.
	This fact can be checked by mean  of Kato's inequality up to the boundary (see \cite[Lemma 6.11]{po}) applied to the function $u - \overline{u}$.
The proof is complete if we show that  a solution to problem \rife{regpb} does exist. 
	Let us define 
	$$G:L^1(\Omega) \to L^1(\Omega),$$
	$$v \ \ \to \ \ u.$$
	This map assigns to every $v \in L^1(\Omega)$ the solution $u$ to the following linear problem
	$$
	\begin{cases}
	-\Delta u = \tilde{g}(x,v)+\mu &  \text{in}\, \Omega, \\
	u=0 & \text{on}\ \partial \Omega.
	\end{cases}
	$$ 
	Standard  Schauder fixed point theorem applies and one can easily  verifies that a solution $u$ to \rife{regpb} does exist.   In view of what we said before, this concludes the proof of Theorem \ref{methodsubsup}. 
	\end{proof}
	\medskip\medskip
We are now in the position to prove Theorem \ref{main}. 

\begin{proof}[Proof of Theorem \ref{main}]
As we have $fu^{-\gamma} \in L^1(\Omega)$, then the uniqueness is an easy consequence of Proposition 4.B.3 in \cite{bmp}. In fact, if $u_{1}$ and $u_{2}$ are two solutions of \rife{pblaplacian} in the sense of Definition \ref{definition1} with data $\mu_{1}$ and $\mu_{2}$ respectively,  then one has 
$$
\into f\left|\frac{1}{u_{2}^{\gamma}}-\frac{1}{u_{1}^{\gamma}}\right|\leq \into d|\mu_{1}-\mu_{2}|\,,
$$
from which uniqueness is deduced as $f >0$ in ${\Omega}$. 

\medskip

In order to prove existence we apply Theorem \ref{methodsubsup}. We need to find both  a subsolution and a supersolution to problem \rife{pblaplacian} in the sense of Definition \ref{subdef}.

We first look for a subsolution. Let us consider the following problem
\begin{equation}
\begin{cases}
\displaystyle -\Delta v = \frac{f(x)}{v^\gamma} &  \text{in}\, \Omega, \\
v=0 & \text{on}\ \partial \Omega.
\label{lazmck}
\end{cases}
\end{equation} 
It is proved in \cite{bo}  the existence of a solution $v \in L^1(\Omega)$ to  problem \rife{lazmck}.
We need a sharp estimate near the boundary for $v$. We consider the following approximating problems
$$
\begin{cases}
\displaystyle -\Delta v_n = \frac{f_n(x)}{(v_n+\frac{1}{n})^\gamma} &  \text{in}\, \Omega, \\
v_n=0 & \text{on}\ \partial \Omega,
\end{cases}
$$
where $f_n$ is the truncation at level $n$ of $f$ (i.e. $f_{n}=T_{n}(f)$).
It was  proven in \cite{bo} that the nondecreasing sequence $v_n$ converges to a solution of problem \rife{lazmck}. By the linear theory, the sequence $v_n$ belongs to $L^{\infty}(\Omega)$.
Also observe that the term $\displaystyle \frac{f_1}{(v_1 +1)^\gamma}$ belongs to $L^{\infty}(\Omega)$ so that we can apply Lemma 3.2 in \cite{bc} in order to obtain that for a.e. $x$ in $\Omega$
$$
\displaystyle \frac{v_1 (x)}{d(x)} \ge C\int_{\Omega}\frac{d (y) f_1 (y)}{||v_1||_{L^{\infty}(\Omega)}+1}dy \  \ge C >0, 
$$  
where $d(x):=d(x,\partial\Omega)$ is the distance function of $x$ from $\partial\Omega$.
Thus, we have
$$
\displaystyle v (x)\ge v_1 (x) \ge C_1d(x), \ \text{a.e. on}\ \Omega.
$$ 
Therefore, as $f\in L^{\infty}(\Omega_{\delta})$, then  ${f}{v^{-\gamma}}\le f{d(x)^{-\gamma}}$ is integrable in $\Omega$ for every $\gamma<1$.
Thus we have that $v >0 \text{ in } \Omega$,  ${f}{v^{-\gamma}}  \in L^1(\Omega)$, and
$$ \displaystyle -\int_{\Omega}v \Delta \varphi = \int_{\Omega} \frac{f\varphi}{v^\gamma}, \ \ \ \forall \varphi \in C^2_0(\overline{\Omega}).$$
\\ \\Since $\mu$ is a nonnegative Radon measure we clearly have
$$ \displaystyle -\int_{\Omega}v \Delta \varphi \le \int_{\Omega} \frac{f\varphi}{v^\gamma} + \int_{\Omega} \varphi d\mu, \ \ \ \forall \varphi \in C^2_0(\overline{\Omega}), \ \varphi \ge 0,$$
that is, $v$ is a subsolution to the problem \rife{pblaplacian}.

We now look for a supersolution of problem \rife{pblaplacian}.  Let $w$ be the solution of 
\begin{equation}
\begin{cases}
\displaystyle -\Delta w = \mu &  \text{in}\, \Omega, \\
w=0 & \text{on}\ \partial \Omega.
\label{sta}
\end{cases}
\end{equation} 
The existence of a positive solution to the problem \rife{sta} is classical (see for instance \cite{s} where the solution is obtained by duality in a more general framework). 

Let us define $z:=w+v$ where $v$ is again the solution to \rife{lazmck}. We have
$$ \displaystyle -\int_{\Omega}z \Delta \varphi =  - \int_{\Omega}w \Delta \varphi -\int_{\Omega}v \Delta \varphi = \int_{\Omega} \varphi d\mu + \int_{\Omega} \frac{f\varphi}{v^\gamma}, \ \ \ \forall \varphi \in C^2_0(\overline{\Omega}).$$
Recalling that $w$ is nonnegative, then we have $z^\gamma\ge v^\gamma> 0 $. Thus, we can say
$$ \int_{\Omega} \varphi d\mu + \int_{\Omega} \frac{\varphi}{v^\gamma} \ge \int_{\Omega} \varphi d\mu + \int_{\Omega} \frac{f\varphi}{z^\gamma}, \ \ \ \forall \varphi \in C^2_0(\overline{\Omega}), \ \varphi \ge 0\,,$$
that is,  $z$  is a positive function   in $L^1(\Omega)$ such that  $\displaystyle {z^{-\gamma}} \le {v^{-\gamma}} \in L^1(\Omega)$ and
$$ \displaystyle -\int_{\Omega}z \Delta \varphi \ge   \int_{\Omega} \frac{f\varphi}{z^\gamma}+\int_{\Omega} \varphi d\mu, \ \ \ \forall \varphi \in C^2_0(\overline{\Omega}), \ \varphi \ge 0.$$
Thus, $z$ is a supersolution to \rife{pblaplacian}.
We can then apply  Theorem \ref{methodsubsup} to conclude that there exists a solution $u$ to problem \rife{pblaplacian} in the sense of Definition \ref{defmain}. 
\end{proof}

\begin{remark}
We stress that strict positivity of $f$ in $\Omega$ is only used to deduce uniqueness of a solution. 
For a  purely nonnegative $f$,  then the existence of a solution can be deduced exactly with the same argument.   
\end{remark}

\section{Nonlinear principal part}
\label{secintrop}

In this last section we show how the existence results proved before can be extended to the case of a nonlinear principal part. For a given real number $p$ with $2-\frac{1}{N}<p< N$, let $a:\Omega \times  \rn \to \rn$ be a Carath\'eodory function such that there exist $\alpha,\beta >0$ with  
\begin{equation}
(a(x,\xi)-a(x,\xi^*))\cdot(\xi-\xi^*)>0,
\label{cara1}
\end{equation}
\begin{equation}
a(x,\xi)\cdot\xi>\alpha|\xi|^p,
\label{cara2}
\end{equation}
\begin{equation}
a(x,\xi)\le \beta(c(x) + |\xi|^{p-1}),
\label{cara3}
\end{equation}
for every $\xi,\xi^* \in \mathbb{R}^N$ such that $\xi \neq \xi^*$, for almost every $x$ in $\Omega$, and  $c(x)$ belongs to $L^{p'}(\Omega)$. The operator 
$$
- {\rm div} (a (x,\nabla u))
$$
is a classical Leray-Lions type operator which maps continuously $W^{1,p}_{0}(\Omega)$ into its dual $W^{-1,p'}(\Omega)$ whose simplest model is the $p$-laplacian (i.e. $a(x,\xi)=|\xi|^{p-2}\xi$). 

We will prove existence of a nonnegative weak solution to the following problem 
\begin{equation}
\begin{cases}
 \displaystyle -\operatorname{div}(a(x,\nabla u)) = \frac{f(x)}{u^\gamma}+\mu &  \text{in}\, \Omega, \\
 u=0 & \text{on}\ \partial \Omega,
\label{pbnl}
\end{cases}
\end{equation}
where $\Omega$ is an open bounded subset of $\rn$, $N\ge 2$, $\gamma > 0$, $f$ is a nonnegative function which belongs to $L^1(\Omega)$ and $\mu$ is a nonnegative bounded Radon measure. 

The bound from below for $p$ is a classical technical assumption that guarantees, even for nonsmooth data, that the gradient of the solutions will belong at least to $(L^{1}_{loc}(\Omega))^{N}$.   

Here is  the suitable definition for  weak solutions to \rife{pbnl}.
\begin{defin}
If $\gamma\leq 1$, a weak solution to problem \rife{pbnl} is a function $u\in W^{1,1}_0(\Omega)$  such that 
 \begin{equation} \displaystyle \forall \omega \subset\subset \Omega \ \exists c_{\omega} : u \ge c_{\omega}>0,\label{weakdefnl1}\end{equation}
and such that
 \begin{equation} \displaystyle \int_{\Omega}a(x,\nabla u) \cdot \nabla \varphi =\int_{\Omega} \frac{f\varphi}{u^\gamma} + \int_{\Omega} \varphi d\mu,\ \ \ \forall \varphi \in C^1_c(\Omega).\label{weakdefnl2}\end{equation}
 If $\gamma>1$  a weak solution for problem \rife{pbnl} is a function $u\in W^{1,1}_{loc}(\Omega)$ such that \rife{weakdefnl1} and \rife{weakdefnl2} are satisfied and $T_k^{\frac{\gamma-1+p}{p}}(u)$ belongs to $W^{1,p}_0(\Omega)$ for every fixed $k>0$.
\label{weakdefnl}
\end{defin}

\begin{theorem}
Let $\gamma>0$. Then there exists a weak solution $u$ to \rife{pbnl} in the sense of Definition \ref{weakdefnl}. Moreover:
 \begin{itemize}
 \item [(i)] if $\gamma\leq 1$, then  $u\in W^{1,q}_{0}(\Omega)$ for every $q<\frac{N(p-1)}{N-1}$, 
\item [(ii)] if $\gamma>1$, then $u\in W^{1,q}_{loc}(\Omega)$ for every $q<\frac{N(p-1)}{N-1}$.   \end{itemize}
\label{mainteo}
\end{theorem}

\begin{proof}
The proof of Theorem \ref{mainteo} strictly follows the main steps of the previous section. We will then sketch it by enlightening the main differences. Estimates will essentially involve ellipticity and so they will be formally very similar to the ones in Section  \ref{secgen}. The main issue in this case will be to pass to the limit in the principal part of the approximating solutions for which we will need to prove the almost everywhere convergence of the gradients that will be based on monotonicity arguments relying on \rife{cara1}.  

\medskip

{\bf Step $1$. Existence for the  approximating  problems.} 
 Let us consider the following problem
\begin{equation}
\begin{cases}
 \displaystyle -\operatorname{div}(a(x,\nabla u_n)) =\frac{f_n}{(u_n+ \frac{1}{n})^\gamma}+\mu_n &  \text{in}\, \Omega, \\
 u_n=0 & \text{on}\ \partial \Omega,
\label{eqapproxnp}
\end{cases}
\end{equation}
where, as before, $f_n$ is the truncation at level $n$ of $f$ and $\mu_n$ is a sequence of smooth functions bounded in $L^1(\Omega)$ and  converging weakly to $\mu$ in the sense of the measures. 

The proof of existence  of a weak solution of problem \rife{eqapproxnp} for every fixed $n \in \mathbb{N}$ is formally identical to the one in Lemma \ref{soleqapprox}. We define the operator 
$$G:L^p(\Omega) \to L^p(\Omega)\,,$$
that  assigns to every $v \in L^p(\Omega)$ the solution $w$ to the following problem
$$
\begin{cases}
 \displaystyle -\operatorname{div}(a(x,\nabla w)) =\frac{f_n}{(|v|+ \frac{1}{n})^\gamma}+\mu_n &  \text{in}\, \Omega, \\
 w=0 & \text{on}\ \partial \Omega\,.
\end{cases}
$$
A straightforward re-adaptation of  the proof of Lemma \ref{soleqapprox} allows us to prove both that the ball of radius $C'\left (n^{\gamma+1}+C(n)\right)$ is invariant for $G$ and the set $G(L^p(\Omega))$ is relatively compact in $L^p(\Omega)$. 
Concerning the continuity of $G$ we use monotonicity.   Let us choose a sequence $v_k$ that converges to $v$ in $L^p(\Omega)$, then we need to prove that $G(v_k)$ converges to $G(v)$ in $L^p(\Omega)$. By compactness we already know that the sequence $G(v_k)$ converges to some function $w$ in $L^p(\Omega)$. We only need to prove that $w=G(v)$. This means that we need to pass to the limit with respect to $k$ in the following weak formulation

\begin{equation}
 \displaystyle \int_{\Omega}a(x,\nabla w_k) \cdot \nabla \varphi =\int_{\Omega} \frac{f_n\varphi}{(|v_k|+\frac{1}{n})^\gamma} + \int_{\Omega}  \mu_n \varphi,
\label{schauder1}
\end{equation}
where $\varphi \in W^{1,p}_0(\Omega)$ and $w_k=G(v_k)$.\\
All terms but the one on the left hand side of \rife{schauder1} pass to the limit. 
Concerning the term on the left hand side we need to check  the almost everywhere convergence of the gradients of $w_k$.

We take $w_k-w$ as test function in the weak formulation of \rife{eqapproxnp}
\begin{eqnarray}
\label{schauder2}
 \displaystyle \int_{\Omega}(a(x,\nabla w_k)-a(x,\nabla w)) \cdot \nabla (w_k-w) =\int_{\Omega} \frac{f_n(w_k-w)}{(|v_k|+\frac{1}{n})^\gamma}    \\ \nonumber + \int_{\Omega}  \mu_n (w_k-w) - \int_{\Omega}a(x,\nabla w) \cdot \nabla (w_k-w).
\end{eqnarray}
Since $w_k$ converges to $w$ in $L^p(\Omega)$, then the first and the second term on the right hand side of \rife{schauder2} tends to zero when $k$ tends to infinity. 
Also the third term tends to zero since, by classical theory (see for instance \cite{ll}), we have that $w_k-w$ weakly converges to zero in $W^{1,p}_{0}(\Omega)$ and $a(x,\nabla w)$ is a fixed function in $(L^{p'}(\Omega))^{N}$. This means that the term on the left hand side of \rife{schauder2} tends to zero so that we can apply Lemma 5 in \cite{bmp1} to obtain that $\nabla w_k$ converges almost everywhere to $\nabla w$. This means that $w=G(v)$.
We can then  apply Schauder fixed point theorem and maximum principle in order to get the existence of a nonnegative solution  in $W^{1,p}_0(\Omega)$ for problem \rife{eqapproxnp}. Moreover, by classical regularity theory (see for instance \cite{bmp1}), $u_n$ belongs to $L^\infty(\Omega)$. 

\medskip

{\bf Step $2$. Local uniform bound from below.} 
Here we  show that $u_n$ is bounded from below on the compact subsets of $\Omega$. 
In particular, we want  to check that the sequence $u_n$ is such that for every $\omega \subset\subset \Omega$ there exists $c_{\omega}$ (not depending on $n$) such that 
\begin{equation} \label{stimola} \displaystyle u_n(x) \ge c_{\omega}>0,\text{   for a.e. $x$ in $\omega$, for every n in $\mathbb{N}$.}\end{equation}

We consider the sequence of problems
\begin{equation}
\begin{cases}
 \displaystyle -\operatorname{div}(a(x,\nabla v_n)) =\frac{f_n}{(v_n+ \frac{1}{n})^\gamma} &  \text{in}\, \Omega, \\
 v_n=0 & \text{on}\ \partial \Omega.
\label{eqdcapprox}
\end{cases}
\end{equation}
It was proved in \cite{dc} the existence of a weak solution for \rife{eqdcapprox} such that
$$ \displaystyle \forall \omega \subset\subset \Omega \ \exists c_{\omega} : v_n \ge c_{\omega}>0,$$
for almost every $x$ in $\omega$ and where $c_{\omega}$ is indipendent on $n$. 
\\ Thus, taking $(u_n-v_n)^-$ as test function in the difference between  \rife{eqapproxnp} and \rife{eqdcapprox}, we obtain
$$\begin{array}{l} \displaystyle \int_{\Omega}(a(x,\nabla u_n)-a(x,\nabla v_n))\cdot\nabla (u_n-v_n)^- \\\\ \dys = \int_{\Omega} \left(\frac{f_n(x)}{(u_n+\frac{1}{n})^\gamma}-\frac{f_n(x)}{(v_n+\frac{1}{n})^\gamma}\right)(u_n-v_n)^- + \int_{\Omega} \mu_n(u_n-v_n)^- \ge 0,\end{array}$$
so that, by  \rife{cara1} 
$$ \displaystyle 0 \ge -\int_{\{u_n\le v_n\}} (a(x,\nabla u_n)-a(x,\nabla v_n))\cdot\nabla (u_n-v_n) \ge 0.$$
This implies, by monotonicity, that
$$\nabla (u_n-v_n)^-=\nabla (u_n-v_n)\chi_{\{u_n\le v_n \} }=0\,,$$
so that, 
$$v_{n}\leq u_{n}\,,$$
and so
$$
 \displaystyle \forall \omega \subset\subset \Omega \ \exists c_{\omega} : u_n \ge v_n \ge c_{\omega}>0,
$$
for almost every $x$ in $\omega$.

{\bf Step $3$. Estimates on the approximating solutions.} 
Here we look for some estimates on $u_n$ in some Sobolev spaces.  As in the semilinear case these estimates will depend on the value of $\gamma$. 

First of all we introduce an auxiliary function that will be useful for our purpose
$$S_{k}(s) = \begin{cases} 1& s> k+1,\\
							s-k& k< s\le k+1,\\
							0& 0\le s \le k.\\
 \end{cases}$$
We observe that if $s\ge 0$ then $S_{k}(s)\le s$.

{\bf The case $\gamma \leq 1$.} We will prove that $u_n$ is bounded in $W^{1,q}_0(\Omega)$ for every $q<\frac{N(p-1)}{N-1}$.

As in the semilinear case we first  observe that the truncations of the approximating solutions are bounded in the energy space $W^{1,p}_{0}(\Omega)$. In fact, we take $(T_k(u_n) +\epsilon)^\gamma - \epsilon^\gamma$ as test function in the weak formulation of \rife{eqapproxnp} where $\epsilon$ is a fixed number strictly smaller than $\frac{1}{n}$ (the case $\gamma=1$ implies the obvious simplifications).  
Reasoning as in the proof of Lemma \ref{bounded<1} we readily obtain
 $$ \displaystyle \int_{\Omega}|\nabla T_k(u_n)|^p \le C(k^{1-\gamma}+k).$$

In particular 
$$
 \displaystyle \int_{\Omega}|\nabla T_1(u_n)|^p \le C\,.
$$

In order to get an estimate for $u_{n}$ in  $W^{1,q}_0(\Omega)$ for every $q<\frac{N(p-1)}{N-1}$ it suffices  to look for this bound on $G_{1}(u_{n})$.

We take $T_k(G_1(u_n))$ as test function in the weak formulation of \rife{eqapproxnp} for $k>1$. Then we have 
$$
 \displaystyle  \int_{\Omega}a(x,\nabla u_n)\cdot\nabla T_k(G_1(u_n)) \le \int_{\Omega} \frac{f_n[(T_k(G_1(u_n))]}{(u_n+\frac{1}{n})^{\gamma}} + \int_{\Omega}  \mu_n T_k(G_1(u_n)) \le Ck,
$$
and also
$$
 \displaystyle  \int_{\Omega}a(x,\nabla u_n)\cdot\nabla T_k(G_1(u_n)) = \int_{\Omega}a(x,\nabla T_k(G_1(u_n)))\cdot\nabla T_k(G_1(u_n)) \ge \alpha \int_{\Omega} |\nabla T_k(G_1(u_n))|^ p,
$$
so that we obtain
\begin{equation}
 \displaystyle \int_{\Omega} |\nabla T_k(G_1(u_n))|^ p \le Ck\,.
\label{stima1nl<1}
\end{equation}
Now we take $S_{k}(G_1(u_n))$ as test function again in the weak formulation of \rife{eqapproxnp}, so that 
$$
 \displaystyle  \int_{\Omega}a(x,\nabla u_n)\cdot \nabla G_1(u_n) S_{k} '(G_1(u_n)) \le C,
$$
and again
$$\begin{array}{l}
 \displaystyle  \int_{\Omega}a(x,\nabla u_n)\cdot \nabla G_1(u_n)S_{k} '(G_1(u_n)) \\ \\ \dys = \int_{\{k\le G_1(u_n) \le k+1\}} a(x,\nabla G_1(u_n))\cdot \nabla G_1(u_n) \ge \int_{\{k\le G_1(u_n) \le k+1\}} |\nabla G_1(u_n)|^p,
\end{array}$$
and finally
\begin{equation}
 \displaystyle \int_{\{k\le G_1(u_n) \le k+1\}} |\nabla G_1(u_n)|^p \le C.
\label{stima2nl<1}
\end{equation}
Thanks to estimates  \rife{stima1nl<1} and  \rife{stima2nl<1}  we can proceed as in Section II. 4 of \cite{bg} in order to obtain that $G_1(u_n)$ is bounded in $W^{1,q}_0(\Omega)$ for every $q<\frac{N(p-1)}{N-1}$. 
Thus we can conclude  that $u_n$ is bounded in $W^{1,q}_0(\Omega)$ for every $q<\frac{N(p-1)}{N-1}$.

{\bf The case $\gamma>1$.} As in  the semilinear case, here we only look for local estimates.
First of all we observe that the (global) bound on $G_1(u_n)$ in $W^{1,q}_0(\Omega)$ for every $q<\frac{N(p-1)}{N-1}$ follows exactly as in the the case $\gamma\leq 1$. 

What is left is the proof that  $T_k(u_n)$ is bounded in $W^{1,p}_{loc}(\Omega)$ for every $k>0$. We take $\varphi = T_k^\gamma(u_n)$  as test function in the weak formulation of \rife{eqapproxnp} in order to have 
$$
 \displaystyle  \int_{\Omega}a(x,\nabla u_n)\cdot \nabla T_k(u_n) T_k^{\gamma-1}(u_n)\le C + Ck^\gamma,
$$
so that using  \rife{stimola} we obtain
$$
 \displaystyle  \int_{\Omega}a(x,\nabla u_n)\cdot \nabla T_k(u_n) T_k^{\gamma-1}(u_n)\ge \alpha \int_{\Omega} |\nabla T_k(u_n)|^p T_k^{\gamma-1}(u_n) \ge C\int_{\omega} |\nabla T_k(u_n)|^p,
$$
where $\omega$ is an arbitrary compact subset of $\Omega$. Thus, we have 
$$
 \displaystyle \int_{\omega} |\nabla T_k(u_n)|^ p \le C + Ck^\gamma,
$$
for every $\omega \subset\subset \Omega$ and every $k>0$.
\\ Combining the estimates on both $G_{1}(u_{n})$ and $T_{1}(u_{n})$ we deduce that $u_{n}$ is uniformly bounded in $W^{1,q}_{loc}(\Omega)$ for every $q<\frac{N(p-1)}{N-1}$.

In order  to give sense to the function on the boundary of $\Omega$ we reason as before, we take 
$\varphi = T_k^\gamma(u_n)$  as test function in the weak formulation of \rife{eqapproxnp} and we use ellipticity in order to get 
$$
 \displaystyle  \int_{\Omega} |\nabla T_k^{\frac{\gamma-1+p}{p}}(u_n)|^p\leq C(1+k^{\gamma})\,.
$$

Thanks to the estimates we proved here we readily deduce, by compact embeddings, the existence of a function $u$  such that (up to not relabeled subsequences) $u_{n}$ converges to $u$ a.e.,  strongly  in $L^{1}(\Omega)$ and weakly in $W^{1,q}_{0}(\Omega)$ ($W^{1,q}_{loc}(\Omega)$ if $\gamma>1$) for every $q<\frac{N(p-1)}{N-1}$. In particular  $a(x,\nabla u_n)$ is bounded in $W^{1,q}_{loc}(\Omega)$ for every $q<\frac{N}{N-1}$. So that in order to pass to the limit in \rife{eqapproxnp} and to conclude the proof of Theorem \ref{mainteo} we only need to check the a.e. convergence of $\nabla u_{n}$ towards $\nabla u$.

{\bf Step $4$. The a.e. convergence of the gradients.} The a.e. convergence of $\nabla u_{n}$ towards $\nabla u$ follows in a standard way if  we prove that  $\nabla T_k(u_n)$ converges to $\nabla T_k(u)$ in $L_{loc}^q(\Omega)$ for every $q<p$, for every $k>0$.

By Definition 2.29 and Remark 2.32 in \cite{dmop} we know that $T_k(u_n)$ is such that
\begin{equation}
\displaystyle -\operatorname{div}(a(x,\nabla T_k(u_n))) =\left(\frac{f_n}{(|u_n|+ \frac{1}{n})^\gamma} + \mu_n\right)\chi_{\{|u_n|\le k\}} + \lambda_{n,k} \ \ \   \text{in}\, \Omega,
\label{aeconvnl}
\end{equation}
where $\lambda_{k,n}$ is a nonnegative diffuse measure (i.e. it is an absolutely continuous measure with respect to the $H^1$-capacity) concentrated on the set $\{u_n=k\}$.

The first term on the right hand side of \rife{aeconvnl} is bounded in $L^1_{loc}(\Omega)$. 

To bound the second term we take $T_k^\gamma(u_n)$ as a test function in the weak formulation of \rife{eqapproxnp} to have 
\begin{equation}
\int_{\Omega}a(x,\nabla u_n)\cdot\nabla T_k(u_n) T_k^{\gamma-1}(u_n) \le C||f||_{L^1(\Omega)} + k^\gamma||\mu_n||_{L^1(\Omega)}.
\label{aeconv1nl}
\end{equation} 
Now we take $T_k^\gamma(u_n)$ as a test function in the weak formulation of \rife{aeconvnl} 
\begin{equation}
\begin{array}{l}
\dys \int_{\Omega}a(x,\nabla T_k(u_n))\cdot\nabla T_k(u_n) T_k^{\gamma-1}(u_n) \\ \\ \dys = \int_{\{|u_n|\le k\}}\left(\frac{f_n}{(|u_n|+ \frac{1}{n})^\gamma} + \mu_n\right)T_k^\gamma(u_n) + k^\gamma\lambda_{n,k}(\Omega).
\label{aeconv2nl}\end{array}
\end{equation}  
Since the first term on the right hand side of \rife{aeconv2nl} is positive then the estimates \rife{aeconv1nl} and  \rife{aeconv2nl} imply 
$$\lambda_{n,k}(\Omega) \le \frac{C}{k^\gamma},
$$
so that $\lambda_{n,k}$ is uniformly bounded in as a measure with respect to $n$, for every fixed $k\geq 1$.
 Due to these bounds on the right hand side of \rife{aeconvnl}, as $T_k(u_n)$ is bounded in $W^{1,p}_{loc}(\Omega)$,  we can proceed as in the proof of Theorem 2.1 in \cite{bm}. In fact, one can obtain that
$$
\limsup_{n} \int_{\omega} (a(x,\nabla T_{k}(u_{n}))-a(x,\nabla T_{k}(u)))\cdot\nabla T_{h}(T_{k}(u_{n})-T_{k}(u))\leq C_{\omega,k}h\,,
$$
for every $\omega\subset\subset\Omega$ and $h>0$.  The previous estimate is known to imply (see again  \cite{bm})  that $\nabla T_k(u_n)$ converges to $\nabla T_k(u)$ in $L_{loc}^q(\Omega)$ for every $q<p$. 
\bk

\end{proof}

\begin{remark}
We point out  that, for the sake of exposition, we assumed that the operator $a$ was  chosen to be independent of  $u$. Anyway one can easily realize that the same proof can be straightforwardly extended to more general Leray-Lions operator involving Carath\'eodory functions $a:\Omega \times  \mathbb{R} \times \rn \to \rn$ such that
$$
(a(x,s,\xi)-a(x,s,\xi^*))\cdot(\xi-\xi^*)>0,
$$
$$
a(x,s,\xi)\cdot\xi>\alpha|\xi|^p,
$$
$$
a(x,s,\xi)\le \beta(c(x) + |s|^{p-1}\bk + |\xi|^{p-1}),
$$
for every $\xi,\xi^* \in \mathbb{R}^N$ such that $\xi \neq \xi^*$, for almost every $x$ in $\Omega$, $s\in \mathbb{R}$,  $0<\alpha\leq \beta$, and  $c(x)$ in $L^{p'}(\Omega)$.
\end{remark}

\subsection*{Acknowledgement}
The authors would like to thank the anonymous reviewers whose comments and suggestions helped to improve the final version of this manuscript. 
Also, the authors  are  partially supported by  the Gruppo Nazionale per l'Analisi Matematica, la Probabilit\`a e le loro Applicazioni (GNAMPA) of the Istituto Nazionale di Alta Matematica (INdAM).

\end{document}